\documentclass[12pt]{article}
\usepackage[utf8]{inputenc}
\usepackage{amsmath,amsfonts,amssymb,amsthm}
\usepackage{bbm}
\usepackage{mathtools}
\usepackage{graphicx}
\mathtoolsset{showonlyrefs}
\usepackage{tabularx}
\usepackage{hyperref}
\usepackage{color}
\usepackage{csquotes} 
\definecolor{darkgreen}{rgb}{0,0.55,0}
\newcommand{\grad}{\nabla}

\newcommand{\laplace}{\Delta}

\newcommand{\N}{\mathbbm{N}}

\newcommand{\R}{\mathbbm{R}}

\usepackage{sectsty}

\sectionfont{\large}
\subsectionfont{\normalsize}
\subsubsectionfont{\normalsize}
\paragraphfont{\normalsize}

\newcommand{\eps}{\varepsilon}

\def\XXint#1#2#3{{\setbox0=\hbox{$#1{#2#3}{\int}$ }
		\vcenter{\hbox{$#2#3$ }}\kern-.59\wd0}}

\newtheorem{prop}{Proposition}
\newtheorem{theorem}{Theorem}
\newtheorem{lemma}{Lemma}

\newtheorem{cor}{Corollary}

\begin{document}
	
	\phantom{ }
	\vspace{4em}
	
	\begin{flushleft}
		{\large \bf On the dynamics of vortices in viscous 2D flows}\\[2em]
		{\normalsize \bf Stefano Ceci and Christian Seis}\\[0.5em]
		\small Institut f\"ur Analysis und Numerik,  Westf\"alische Wilhelms-Universit\"at M\"unster, Germany.\\
		E-mails: ceci@wwu.de, seis@wwu.de \\[3em]
		{\bf Abstract:} 
		We study the 2D Navier--Stokes solution starting from an initial vorticity mildly concentrated near $N$ distinct points in the plane. We prove quantitative estimates on the propagation of concentration near a system of interacting point vortices introduced by Helmholtz and Kirchhoff. Our work extends the previous results in the literature in three ways: The initial vorticity is concentrated in a weak (Wasserstein) sense, it is merely $L^p$ integrable for some $p>2$, and the estimates we derive are uniform with respect to the viscosity.\\[3em]
		{\bf Statements and Declarations:} On behalf of all authors, the corresponding author states that there is no conflict of interest. Data sharing not applicable to this article as no datasets were generated or analysed during the current study.
		
	\end{flushleft}
	
	\vspace{2em}

	\section{Introduction}
	
	Coherent flow structures are widely observed in two-dimensional turbulent fluid motions. Isolated regions of concentrated vorticity emerge, for instance, as a result of stirring or solid body interactions, and they persist for rather long time scales while moving with the ambient background flow and interacting with nearby vortex regions.
	Vortex regions that are sufficiently far from each other can be considered as rigid disc-like objects, whose interactions reduce approximately to the interactions of their centers.
	
	The idea of studying the motion of vortices in two-dimensional fluid flows by means of an idealized point vortex system dates back to the second half of the nineteenth century and the work of Helmholtz \cite{Helmholtz1858}. Arguing formally, Helmholtz derived an ODE model for interacting point vortices in inviscid incompressible fluids. Kirchhoff later documented the Hamiltonian structure of this system \cite{Kirchhoff76}, and, as a result, the ODE model is nowadays commonly referred to as the Helmholtz--Kirchhoff point vortex system.
	
	The Helmholtz--Kirchhoff system is a reasonable description of the actual vortex dynamics of isolated and highly concentrated vortex regions. In these situations, if the fluid kinematic viscosity is small, it seems quite natural to expect that the vortices will exhibit a behavior similar to the Helmholtz--Kirchhoff idealization. Likewise, it is reasonable to expect that this description will become invalid in the event of collisions and merging of vortices or if vortex regions dissolve in the ambient background flow as a result of viscous friction. The goal of the present paper is to derive estimates that describe to what extent the point vortex system captures the motion of vortex regions in incompressible viscous flows.
	
	For inviscid fluids, the first rigorous work in this direction is due to Marchioro and Pulvirenti \cite{MarchioroPulvirenti83}, who proved that bounded vortex patch solutions to the Euler equation remain close to the point vortices during the evolution. Their strategy exploits the regularity of the velocity field sufficiently far from the vortex patches and the symmetry properties of the two-dimensional Biot--Savart kernel. The original work was gradually improved in many subsequent articles, see for example \cite{MarchioroPulvirenti93, Marchioro98, ButtaMarchioro18}.
	Other methods have been employed to study the connection of the point vortex system with the Euler equation. We mention the approach by Turkington \cite{Turkington87}, which relies on the conservation of the energy. This has the advantage that it works also in situations where the Biot--Savart operator does not possess certain symmetry features, for example in the setting of three-dimensional axisymmetric fluids without swirl, where it has been successfully applied \cite{BenedettoCagliotiMarchioro00}. It does not seem however suited to studying several interacting vortices; in this case a combination with the Marchioro--Pulvirenti method has been used \cite{ButtaMarchioro19, ButtaCavallaroMarchioro21}. Another technique effectively implemented to this context is the so-called gluing method \cite{DavilaDelPinoMussoWei18, DavilaDelPinoMussoWei20}. In these works, solutions with highly concentrated vorticities containing precise information on the vortex cores are constructed in the two-dimensional case, and helical vortex filaments
	are built in the three-dimensional setting.
	
	In the context of viscous fluids, the connection with the Helmholtz--Kirchhoff point vortices has been shown for the Navier--Stokes equation \cite{Marchioro90, Marchioro98_2, CetroneSerafini18}, where the authors consider bounded vorticities initially   sharply concentrated in separated regions and prove that they converge in the inviscid limit to the point vortices. 
	A more singular initial datum was considered by Gallay \cite{Gallay11}, namely a collection of several point vortices. He proved that the corresponding Navier--Stokes vorticity concentrates, in terms of a weighted $L^2$ norm, near a combination of Lamb--Oseen vortices centered around a viscous regularization of the point vortex system (see also the nice review paper \cite{Gallay12}). A Lamb--Oseen vortex (see expression \eqref{3.lo}) is the solution corresponding to an initial vorticity concentrated on a single point, and it can in some ways be regarded as the fundamental solution to the two-dimensional Navier--Stokes vorticity equation. Gallay gives an estimate on the rate of concentration, proving that it is proportional to $\nu t$, where $\nu$ is the viscosity and $t$ is the time. As consequence of this result, the convergence of the Navier--Stokes solution to the Helmholtz--Kirchhoff point vortex system for $\nu \to 0$ is also established.

It would be desirable to establish a relation between solutions to the Navier--Stokes equations and the (viscous) point vortex dynamics that has both features: \emph{First}, it provides  stability estimates in the sense of the Marchioro--Pulvirenti strategy showing that solutions that start close to point vortices remain close to vortices. \emph{Second}, it describes the accurate shape of vortices that are actually spreading as a result of viscosity.

In the present work, we take a tiny step towards this goal by addressing the first property and by improving in several  aspects  the aforementioned works \cite{Marchioro90, Marchioro98_2} by Marchioro on the Navier--Stokes equations:
		As a measure for vortex concentration, we work with Wasserstein distances. Weak notions of concentration are suitable in the viscous setting due to the fact that the viscosity spreads instantaneously over the full space. However, also in the inviscid case, weak notions of concentration are favorable as they allow for more general vortex configurations like elongated vortex regions or even long tentacles, as can be observed in chaotic or turbulent flows. We have considered concentration in the Wasserstein sense in our earlier works on the Euler equation \cite{CeciSeis21, CeciSeis21b}. In  the context of three-dimensional Euler filaments, geometric flat norms, which are related to Wasserstein distances in 2D, were considered \cite{JerrardSeis17}.

	Furthermore, we establish estimates that hold uniformly in the viscosity constant, independently from the scale of concentration. This way our results apply also to initial data that are given by a collection of point vortices as studied by Gallay \cite{Gallay11}, and our estimates describe the sharp rates of vortex spreading due to viscosity.
		
	We finally consider configurations whose  initial vorticity is allowed to be unbounded but in $L^p$, for some $p>2$, with almost no assumption on the magnitude of the $L^p$ norm. This same condition was studied \cite{CeciSeis21b} in the Euler setting and extends the corresponding requirement on the $L^\infty$ norm that was considered in \cite{Marchioro98_2, CetroneSerafini18}. 
 
	%

%
\medskip

\emph{Organization of the paper.}
	In Section \ref{Sec2} we introduce precisely the mathematical setting and state the main results. In Section \ref{Sec3} we present the proofs.

	\section{Mathematical setting}\label{Sec2}
	
	We study the dynamics of vortices in a viscous incompressible fluid in the two-dimensional plane $\R^2$. This motion can be modelled via the Navier--Stokes equations in vorticity form. These read
	\begin{equation}\label{3.1}
		\partial_t \omega + u\cdot \grad \omega = \nu \laplace\omega \quad\text{in }(0,+\infty)\times\R^2,
	\end{equation}
	where $u:(0,+\infty)\times\R^2\to\R^2$ is the velocity field and $\omega:(0,+\infty)\times\R^2\to\R$ is the scalar vorticity field that measures the tendency of the fluid to rotate. The incompressibility condition is translated into the mathematical constraint that the velocity field has zero divergence,
	\begin{equation}\label{3.48}
		\nabla\cdot u=0 \quad\text{in }(0,+\infty)\times\R^2.
	\end{equation}
	The vorticity field can be computed from the velocity as the rotation $\omega = \partial_1 u_2 - \partial_2 u_1$, and vice versa the velocity $u$ can be recovered from $\omega$ via the Biot--Savart law
	\begin{equation}\label{3.bs}
		u(t,x) = K*\omega(t,x) = \int_{\mathbb{R}^2} K(x-y)\,\omega(t,y)\,dy,
	\end{equation}
	where the convolution is understood in space. Here, the function $K$ is the Biot--Savart kernel
	\begin{equation*}
		K(z) = \frac{1}{2\pi}\frac{z^\perp}{|z|^2},
	\end{equation*}
	and $z^\perp$ denotes the counter-clockwise rotation of the vector $z$ by $90$ degrees. Finally, the constant $\nu>0$ is the kinematic viscosity of the fluid.
	
	We consider a compactly supported initial datum $\bar{\omega}\in L^p(\R^2)$ for some $p>2$. It is well-known (see \cite{Gallay12} and references therein) that there exists a unique solution $\omega\in C\big( (0,+\infty); L^1\cap L^\infty(\mathbb{R}^2) \big)$ to equation \eqref{3.1}.
	Such a solution has been shown to be smooth in space and time \cite{BenArtzi94}, and its Lebesgue norms are non-increasing in time, that is
	\begin{equation*}
		\|\omega(t)\|_{L^1} \le \|\bar{\omega}\|_{L^1}, \quad \|\omega(t)\|_{L^p} \le \|\bar{\omega}\|_{L^p}.
	\end{equation*}
	
	We study the situation where the initial vorticity is split into $N$ components of definite sign
	\begin{equation*}
		\bar{\omega}=\sum_{i=1}^{N}\bar{\omega}_i,
	\end{equation*}
	where for every $i\in\{1,\ldots,N\}$ it holds either $\bar{\omega}_i \ge0$ or $\bar\omega_i\le 0$. The motion of the vorticity components obeys the advection-diffusion equation
	\begin{equation}\label{3.5}
		\begin{cases}
			\partial_t \omega_i + u\cdot \grad\omega_i = \nu \laplace \omega_i , \\
			\omega_i(0,\cdot) = \bar \omega_i(\cdot) .
		\end{cases}
	\end{equation}
	Hence their sign is preserved over time and, because of the uniqueness of the solutions to linear advection-diffusion equations, the vorticity $\omega$ remains expressed as the sum of its components,
	\begin{equation*}
		\omega(t,\cdot) = \sum_{i=1}^{N}\omega_i(t,\cdot).
	\end{equation*}
	We observe that, thanks to the divergence-free condition \eqref{3.48} and to the absence of domain boundaries, the intensity, that is the space integral of each vortex component, is constant in time,
	\begin{equation*}
		a_i = \int_{\R^2}\omega_i(t,x)\,dx = \int_{\R^2}\bar{\omega}_i\,dx.
	\end{equation*}
Depending on the sign of the vorticity components, we will either have that $a_i>0$ or $a_i<0$, so that $\omega_i/a_i$ is a well-defined probability distribution.
	
	We suppose that the vortex components are sharply concentrated around $N$ distinct points $\bar{Y}_1,\ldots,\bar{Y}_N\in\R^2$ in the sense that
	\begin{equation}\label{3.2}
		W_2\left( \frac{\bar{\omega}_i}{a_i},\,\delta_{\bar{Y}_i} \right) \le \varepsilon,
	\end{equation}
	for all $i=1,\ldots,N$. 	The parameter $\eps>0$ is the concentration scale, that will be assumed to be as small as needed. Here $W_2$ is the $2$-Wasserstein distance, which, in the case that one of the measures considered is atomic, is just the square root of the variance,
	\begin{equation}\label{3.55}
		W_2\left( \frac{\bar{\omega}_i}{a_i},\,\delta_{\bar{Y}_i} \right)^2 = \frac{1}{a_i} \int_{\R^2} |x-\bar{Y}_i|^2\,\bar{\omega}_i(x)\,dx,
	\end{equation}
	and measures the average \enquote{size} of the vortex region.
	We refer to Villani's book \cite{Villani03} for more information on Wasserstein distances. To understand the role of $\eps$, notice that condition \eqref{3.2} is satisfied for example if $\bar{\omega}_i$ is supported in a ball of radius $\eps$ centred on $\bar{Y}_i$.
	An easy computation shows that the Wasserstein distance is minimized when the atomic measure is located on the center of vorticity
	\begin{equation*}
		\bar{X}_i = \frac{1}{a_i} \int_{\R^2}x\bar{\omega}_i(x)\,dx,
	\end{equation*}
	that is
	\begin{equation}\label{3.6}
		W_2\left( \frac{\bar{\omega}_i}{a_i},\,\delta_{\bar{X}_i} \right) \le W_2\left( \frac{\bar{\omega}_i}{a_i},\,\delta_{\bar{Y}_i} \right).
	\end{equation}
	
	Since Wasserstein distances metrize the weak convergence in the sense of measures, see Theorem 7.12 in \cite{Villani03}, assumption \eqref{3.2} implies that the  rescaled vortex component $\omega_i/a_i$ converges weakly as $\eps$ goes to zero to an atomic measure. We suppose that the intensities $a_i$ are independent of $\eps$ and, accordingly, the $L^p$ norm of the components must diverge as $\eps\to0$. Here, we assume that
	\begin{equation}\label{3.30}
		\|\bar{\omega}\|_{L^p} \le \eps^{-\gamma},
	\end{equation}
	where $\gamma$ is a fixed positive number. This is the same condition that we considered earlier in the inviscid setting  \cite{CeciSeis21b}, while similar requirements on the $L^\infty$ norm were studied e.g.~in \cite{Marchioro90,Marchioro98_2,CetroneSerafini18} for the Navier--Stokes equations and in \cite{Marchioro98,CapriniMarchioro15} for the Euler equations. Notice that the two assumptions \eqref{3.2} and \eqref{3.30} enforce that $\gamma \ge 2-2/p$ as a consequence of the interpolation $1 \lesssim \|f\|_{L^p}^{p} W_2(f,\delta_X)^{2p-2}$, which holds true for any probability distribution $f$. We furthermore  observe that, since the $L^p$ norm of every component is non-increasing by equation \eqref{3.5}, assumption \eqref{3.30} translates into a condition on each $\omega_i$ for positive times,
	\begin{equation}\label{3.4}
		\|\omega_i(t)\|_{L^p} \le \|\bar{\omega}_i\|_{L^p} \le \|\bar{\omega}\|_{L^p}.
	\end{equation}
	
	Finally, we make a last hypothesis on the initial configuration, namely we suppose that there is not much vorticity far away from their centers. More precisely, we assume that there exists a  radius $R\gg\eps$ 	 and a constant $\beta$  such that
\begin{equation}\label{303}
m_i(0,R) \le \eps^{\beta},
\end{equation}
where $m_i(t,R)$ is the vorticity portion of the $i$-th component at time $t$ lying at least at distance $R$ from its center $X_i(t)$, that is,
	\begin{equation*}
		m_i(t,R)=\frac{1}{a_i}\int_{B_R(X_i(t))^c} \omega_i(t,x)\,dx.
	\end{equation*}
We will assume that
\begin{equation}
\label{317}
 R = R_0 (\eps+\sqrt{\nu})^{\delta},
\end{equation}
for some $\delta\in [0,1/2)$ and some constant $R_0$.

	Our goal is to estimate to what extent the vortex components remain close to the Helmholtz--Kirchhoff point-vortex system starting from $\bar{Y}_1,\ldots,\bar{Y}_N$, that is a collection of points that move according to the equations
	\begin{equation}\label{3.pvs}
		\begin{cases}
			\frac{d}{dt}Y_i(t) = \sum_{\underset{j \neq i}{j=1}}^{N} a_j K(Y_i(t)-Y_j(t)), \\
			Y_i(0) = \bar{Y}_i.
		\end{cases}\quad \forall i=1,\ldots,N.
	\end{equation}
	We will assume that the actual point vortices are not colliding in some time interval $[0,T_c]$, and that their  minimal distance
\[
	 d = \min_{i\not= j}\inf_{t\in[0,T_c]} | Y_i(t)- Y_j(t)|,
\]
		is large compared to $R$, say $d \ge 12 R_0$.

Our first result gives a componentwise estimate.
	
	\begin{theorem}\label{th3.1}
	There exist $\eps_0$, $\nu_0\in(0,1)$, $\beta_0\in(1,\infty)$ and a time $T\le T_c$ such that for any concentration parameter $\eps\in(0,\eps_0)$, viscosity $\nu\in(0,\nu_0)$ and decay parameter $\beta\in(\beta_0,\infty)$ the following holds: 	Let $\bar{Y}_1,\ldots,\bar{Y}_N \in\mathbb{R}^2$ be pairwise distinct points and consider their evolution ${Y}_1,\ldots,{Y}_N \in\mathbb{R}^2$ through the point-vortex system \eqref{3.pvs}. Let $\omega_i$ be an evolving vortex region \eqref{3.5} that is initially   sharply concentrated around the point vortices \eqref{3.2}, decaying away from the centers \eqref{303} and satisfying the bound \eqref{3.30}. Then it holds for all $i=1,\ldots,N$ and $t\in[0,T]$ that
		\begin{equation*}
			W_2 \left( \frac{\omega_i(t)}{a_i},\,\delta_{Y_i(t)} \right) \le C e^{\Lambda t}\,(\varepsilon+\sqrt{\nu t})
		\end{equation*}
where $C$ and $\Lambda$ are  constants that are independent of $\eps$, $\nu$ and $R$. Moreover, $T\ge c$ for some constant $c$ independent of $\eps$, $\nu$ and $R$. If $T_c=\infty$ and $\delta>0$ in \eqref{317}, there is the stronger estimate
\[
T\ge c\log\frac1{\eps+\sqrt{\nu}}.
\]
	\end{theorem}

	The estimate in Theorem \ref{th3.1} features the sharp rate of viscous spreading $\sqrt{\nu t}$:
		In the context of the Navier--Stokes equations, the support of each initial component is instantaneously distributed all over $\R^2$. This effect can be nicely observed when studying the evolution of the self-similar Lamb--Oseen vortex
	\begin{equation}\label{3.lo}
		\omega_{\text{LO}}(t,x) = \frac{1}{\nu t} \Gamma\left(\frac{x}{\sqrt{\nu t}}\right),\quad \Gamma(\xi)  = \frac1{4\pi} e^{-|\xi|^2/4},
	\end{equation}
	which solves the Navier--Stokes equation with initial datum $\delta_0$ (i.e. the Dirac measure concentrated on the origin).
	Apparently, the solution is positive for all positive times, and the vortex core grows in time at the rate $\sqrt{\nu t}$. This growth law can be expressed in terms of the Wasserstein distance,
	\[
	W_2(\omega_{\text{LO}}(t),\delta_0)\sim \sqrt{\nu t},
	\] 
	as can be verified via a straightforward scaling argument. Our result thus  proves that the general ``vortex components'' described initially in  \eqref{3.2} show a similar (optimal) growth. Moreover, in the limit $\eps\to0$, our result describes the optimal spreading rates for the Navier--Stokes equations with atomic initial data.

	Theorem \ref{th3.1} can also be stated as an estimate on the full solution (rather than the single vortex components) in terms of the $1$-Wasserstein distance $W_1$. Thanks to the dual Kantorovich--Rubinstein representation
	\begin{equation*}
		W_1(f,g) = \sup\left\{ \int_{\R^2}(f-g)\zeta\,dx:\;\|\nabla\zeta\|_{L^\infty}\le1 \right\},
	\end{equation*}
	see for example Theorem 1.14 in \cite{Villani03}, $W_1$ is indeed well-defined merely if $f$ and $g$ have the same average or total mass, with no requirement on their sign.
	
	\begin{cor}\label{cor3.1}
		Under the assumptions of Theorem \ref{th3.1}, it holds that
		\begin{equation*}
			W_1 \left( \omega(t), \sum_{i=1}^{N}a_i \delta_{Y_i(t)} \right) \le C e^{\Lambda t} (\eps+\sqrt{\nu t})
		\end{equation*}
		for all $t\le T$.
	\end{cor}
	
	This implies weak convergence in the sense of measures of the Navier--Stokes vorticity $\omega(t)$ to the \enquote{point-vortex measure} $\sum_{i=1}^{N}a_i\delta_{Y_i(t)}$ if we consider the simultaneous limits $\eps\to 0$, $\nu\to 0$, with a convergence rate that, up to an exponential in time, remains of the order of $\eps$, as was assumed for the initial configuration in \eqref{3.2}, plus a correction term $\sqrt{\nu t}$ due to the expansion of the vortices caused by the viscosity.
	
	We remark that the point-vortex system \eqref{3.pvs} can be regarded as a very weak solution to the Euler equation, as was observed by Schochet \cite{Schochet96}.
	Hence, this result may also be interpreted as a stability estimate between the Navier--Stokes and Euler equations.
	Such a bound could be rigorously obtained by a simple application of the triangle inequality when combined with our earlier result \cite{CeciSeis21b}, if suitable estimates on the convergence of viscous to inviscid solutions were available. These would be optimal presumably if an order of $\sqrt{\nu t}$ on the $W_1$ distance between the Navier--Stokes and Euler vorticities could be achieved, but, even in the Yudovich setting of integrable and bounded vorticities, this has not yet been established in general, see e.g. \cite{ConstantinWu95, ConstantinWu96, Chemin96, Seis21}. It is interesting to note that, in the setting considered in this paper, one sees the expected rate $\sqrt{\nu t}$ (at least for viscosities of the order of $\eps^2/t$) in the inviscid limit $\nu\to0$.
	
	We can also express the estimate in terms of the vorticity centers
	\begin{equation*}
		X_i(t) = \frac{1}{a_i} \int_{\R^2}x\omega_i(t,x)\,dx
	\end{equation*}
	of the components and their velocities as follows.
	
	\begin{theorem}\label{th3.2}
		Under the assumptions of Theorem \ref{th3.1}, it holds that
		\begin{equation*}
			|X_i(t)-Y_i(t)| \le C e^{\Lambda t} \, (\varepsilon+\sqrt{\nu t}) \quad\text{and}\quad  \left| \frac{dX_i(t)}{dt}-\frac{dY_i(t)}{dt} \right| \le C e^{\Lambda t} \, (\varepsilon+\sqrt{\nu t})
		\end{equation*}
		for all $t\le T$ and $i=1,\ldots,N$.
	\end{theorem}

We finally comment on our estimate on  time $T$ in Theorem \ref{th3.1}. In view of the exponential growth term in our concentration estimates, the statement of our theorems become meaningless if $T\gg \log\frac1{\eps+\sqrt{\nu}}$. In this regard, the bound on $T$ is satisfactory at least if the ``outer vorticity'' in \eqref{303} vanishes appropriately outside a decreasing ball around the vortex center, i.e., $\delta>0$ in \eqref{317}. The logarithmic bound ceases to hold if we fix a radius independently of $\eps$ and $\nu$, and it is not clear to us, how to overcome this restriction. Comparable logarithmic bounds on  time were obtained earlier by Cetrone and Serafini \cite{CetroneSerafini18} under stronger concentration assumptions. The exponential growth term results from a Gronwall argument. It is very likely that  better estimates on time would require a completely different (and genuinely nonlinear) approach.

	\section{Proofs}\label{Sec3}
	We remark that we treat $\eps$ and $\nu$ as two independent parameters and, since we are interested in what happens when both of them are not too large, in the computations we will always assume that $\eps<1$ and $\nu<1$. Moreover, $\eps$ will always be (sufficiently) smaller than the other length parameters $R$ and $d$, that is $\eps\ll R,\,d$, and we notice  that $R$ is chosen in \eqref{317} such that
		\begin{equation}
	\label{305}
R\le \frac{d}{12}.
	\end{equation}
	Before starting with the proofs, we introduce some notation.
	In the following, $C$ will always denote a positive constant independent of $t$, $\nu$ and $\varepsilon$, possibly depending on other quantities, and whose precise value may change from line to line.  When we write $A \lesssim B$
	we mean that $A \le CB$ for some $C$, and when we write $A \sim B$
	we mean that $A \lesssim B$ and $B \lesssim A$. Hence for example we will often neglect the vortex intensities $a_i$.
	
	We denote the velocity generated by the $i$-th vorticity as $u_i(t,x)=K*\omega_i(t,x)$, where the convolution is meant in space.
	Moreover, we define the ``external field'' $F_i$ acting on $\omega_i$ as the velocity field generated by all other vorticities, that is,
	$F_i = \sum_{j\neq i}u_j$. In this way, the total velocity $u=K*\omega$ satisfies
	\begin{equation*}
		u=\sum_{j=1}^{N}u_j=u_i+F_i.
	\end{equation*}
	The $2$-Wasserstein distance of the $i$-th component from its center will be indicated by
	\begin{equation*}
		W_i(t)=W_2 \left( \frac{\omega_i(t)}{a_i},\,\delta_{X_i(t)} \right) = \sqrt{\frac{1}{a_i}\int_{\R^2} |x-X_i(t)|^2 \omega_i(t,x)\,dx }.
	\end{equation*}

	In our first step, we derive a bound on the growth of the Wasserstein distance.

	\subsection{Proof of the estimate on $W_i$.}
	
	We fix $i\in\{1,\ldots,N\}$.
	It will be convenient to regularize the Biot--Savart kernel by cutting out the singularity at the origin. We do this with the help of 
	a radial cut-off function $\psi\in C^\infty_c(\mathbb{R}^2)$ satisfying
	\begin{equation}\label{3.10}
		0\le\psi\le 1, \quad \psi=1 \text{ on }B_{d/12}(0), \quad \psi=0 \text{ out of } B_{d/6}(0), \quad |\nabla\psi| \lesssim \frac{1}{d}. 
	\end{equation}
We then split the external field acting on $\omega_i$ into two parts, $F_i=F_i^L+F_i^B$, with $F_i^L$ being the regular part,
	\begin{equation*}
		F_i^L(t,x)=\sum_{j\neq i} \int_{\mathbb{R}^2}[1-\psi(x-y)]\,K(x-y)\,\omega_j(t,y)\,dy
	\end{equation*}
	and $F_i^B$ the remainder
	\begin{equation*}
		F_i^B(t,x)=\sum_{j\neq i} \int_{\mathbb{R}^2}\psi(x-y)\,K(x-y)\,\omega_j(t,y)\,dy.
	\end{equation*}
	In the next lemma, we show that $F_i^L$ is indeed a Lipschitz function while  $F_i^B$ is merely bounded.
	
	\begin{lemma}\label{lem3.1}
		$F_i^L$ is Lipschitz in space uniformly in time with Lipschitz constant $\sim 1/d^2$, and $F_i^B$ is bounded with 
		\begin{equation*}
			|F_i^B(t,x)| \lesssim \begin{cases} d^{\theta}\sum_{j\neq i} m_j(t,d/6)^{\theta} \, \|\omega_j(t)\|_{L^p}^{1-\theta} & \quad\text{for }x\in B_{d/6}(X_i),\\
				d^{\frac{p-2}p}\sum_{j\neq i} \|\omega_j(t)\|_{L^p} & \quad\text{for }x\in B_{d/6}(X_i)^c,
			\end{cases}
		\end{equation*}
		for any $t\ge0$ such that
		\begin{equation}\label{3.52}
			\min_{i\not=j} |X_i(t)-X_j(t)|\ge\frac{d}{2},
		\end{equation} 
		where $\theta= (p-2)/(3p-2)$.
	\end{lemma}
	\begin{proof}
		The Lipschitz bound on $F_i^L$ can be easily proven computing the gradient of the integrand,
		\begin{equation*}
			\begin{split}
				|\nabla\big[ \, \big(1-\psi(z) \big) \, K(z) \big]| &\le |K(z)|\,|\nabla\psi(z)| + |1-\psi(z)|\,|\nabla K(z)|  \lesssim \frac{1}{d^2},
			\end{split}
		\end{equation*}
		where we used the scaling of the Biot--Savart kernel in the form $|K(z)|\lesssim 1/|z|$, $|\grad K(z)|\lesssim 1/|z|^2$, and the properties of the cut-off function~\eqref{3.10}.
		
		We concentrate now on proving the boundedness of $F_i^B$. First, if $x\in B_{d/6}(X_i)$, we have for any $y\in B_{d/6}(x)$ and $j\neq i$ that $|y-X_j| \ge |X_j-X_i|-|X_i-x|-|x-y| \ge d/6$ because $|X_i-X_j|\ge d/2 $ by condition \eqref{3.52}. Therefore, for any $\tilde{\theta}\in(0,1)$, it holds that
		\begin{equation}\label{3.11}
			\begin{split}
				|F_i^B(x)| &\lesssim \sum_{j\neq i} \int_{B_{d/6}(x)\cap B_{d/6}(X_j)^c } \frac{1}{|x-y|}\,|\omega_j(y)|^{\tilde{\theta}}\,|\omega_j(y)|^{1-\tilde{\theta}}\,dy \\
				&\le \sum_{j\neq i} \left( \int_{B_{d/6}(x) } \frac{1}{|x-y|^q}\,dy \right)^{\frac{1}{q}}  \left( \int_{ B_{d/6}(X_j)^c } |\omega_j(y)| \,dy \right)^{\tilde{\theta}} \left( \int_{\R^2} |\omega_j(y)|^{p} \,dy \right)^{\frac{1-\tilde{\theta}}{p}},
			\end{split}
		\end{equation}
		where we have used the (generealized) H\"older inequality with exponent $q$  chosen so that $1/q + \tilde{\theta} + (1-\tilde{\theta})/{p} = 1$. Requiring $q<2$ ensures that the first integral in \eqref{3.11} is of the order $d^{(2-q)/q}$. Solving for $\tilde \theta$ gives
		\begin{equation*}
			\tilde{\theta} = \left( 1-\frac{1}{p}-\frac{1}{q} \right) \frac{p}{p-1},
		\end{equation*}
		and thus, the condition that $\tilde{\theta}$ is positive enforces $q \in \left( p/(p-1),2 \right)$. Choosing $q$ to be, say, the mean value between $p/(p-1)$ and $2$ yields $\tilde{\theta}=\theta$ as in the statement of the lemma. 
		
		Second, if $x \notin B_{d/6}(X_i)$, we use the rougher (in fact, globally valid) bound
		\begin{equation*}
		\begin{split}
			|F_i^B(x)| & \lesssim \sum_{j\neq i} \int_{B_{d/6}(x)} \frac{1}{|x-y|}\,|\omega_j(y)|\,dy \\
			&\lesssim \sum_{j\neq i}\left( \int_{B_{d/6}(0)} \frac{1}{|z|^{p'}}\,dz \right)^{\frac{1}{p'}} \|\omega_j\|_{L^p}
			\sim d^{1-\frac2p} \sum_{j\neq i}\|\omega_j\|_{L^p}.
			\end{split}
		\end{equation*}
		This concludes the proof.
	\end{proof}

In principle, we could estimate the $L^p$ norms of the vorticity components in the previous lemma with the help of the a priori estimate \eqref{3.4} and the assumption on the initial data \eqref{3.30}. For large times, however, it is more convenient to make use of the smoothing properties of the diffusive part in \eqref{3.5}. More precisely, we have the following estimate.

\begin{lemma}\label{lem3.1a}
For any $i=1,\dots,N$, $q\in [1,\infty]$ and any $t>0$, it holds that
	\begin{equation}\label{3.14}
		\|\omega_i(t)\|_{L^q} \lesssim \frac{1}{(\nu t)^{1-\frac{1}{q}}}\, \|\bar{\omega_i}\|_{L^1}.
	\end{equation}
\end{lemma}

This estimate is known to be true for solutions to the Navier--Stokes equation (see, e.g., Theorem 4.3 in \cite{GigaMiyakawaOsada88}), but due to possible cancellation effects, we cannot deduce its validity for individual vortex components. Instead, we provide a short proof of this estimate, in which we follow an argumentation given in \cite{FengSverak15}.

\begin{proof}
It is enough to consider the statement for $q=2^k$. The  statement for general $q<\infty$ follows by interpolation, while the statement for $q=\infty$ comes from taking the limit $k\to\infty$. The case $q=1=2^0$ is trivially true since then $\|\omega_i(t)\|_{L^1}=a_i$. We suppose that $k\ge 1$ or, equivalently, $q\ge 2$ from here on.

We consider $E_q(t):= \|\omega_i(t)\|_{L^q}^q$ and observe that it satisfies the identity 
\[
-\frac{d}{dt} E_q(t) = \frac{4\nu (q-1)}{q}\|\grad \omega_i^{q/2}\|_{L^2}^2\sim \nu \|\grad \omega_i^{q/2}\|_{L^2}^2
\]
under the evolution \eqref{3.5}. Notice that we use here the fact that $q\ge 2$, so that the dependency on $q$ in this estimate can indeed be neglected. Making use of the 2D Nash inequality $\|f\|_{L^2}^2\lesssim \|f\|_{L^1}\|\grad f\|_{L^2}$, the latter turns into the estimate
\begin{equation}\label{100}
\frac{d}{dt} E_q(t)^{-1} \gtrsim \nu E_{\frac{q}2}(t)^{-2}.
\end{equation}
It remains to apply an induction argument over $k$, and we suppose thus that the statement holds for $q/2$. Then   \eqref{100} implies
\[
\frac{d}{dt} E_q(t)^{-1} \gtrsim \nu (\nu t)^{q-2} \|\bar \omega_i\|_{L^1}^{-q},
\]
and an integration in time gives
\[
\|\omega_i(t)\|_{L^q}^{-q}  = E_q(t)^{-1} \ge E_q(t)^{-1} - E_q(0)^{-1} \gtrsim (\nu t)^{q-1} \|\bar \omega_i\|_{L^1}^{-q},
\]
from which we easily deduce the statement of the lemma.
\end{proof}

	In the next lemma we use the properties of $F_i^L$ and $F_i^B$ that we just showed to deduce a bound on the evolution of the Wasserstein distance between $\omega_i$ and $X_i$.
	
	\begin{lemma}\label{lem3.2}
		Let $\theta\in(0,1)$ be given as in  Lemma \ref{lem3.1}. 
		For any $t\ge 0$ with property \eqref{3.52}, it holds that
		\begin{equation}\label{3.13}
			\begin{split}
				\frac{d}{dt}W_i^2(t) &\lesssim d^{-2}W_i^2(t) + d^{\theta}\sum_{j\neq i} m_j(t,d/6)^{\theta} \, \|\omega_j(t)\|_{L^p}^{1-\theta}\, W_i(t) \\
				&\quad + d^{\frac{p-2}p} m_i(t,d/6)^{\frac{1}{2}} \, \sum_{j\neq i}\|\omega_j(t)\|_{L^p} \, W_i(t) + \nu.
			\end{split}
		\end{equation}
	\end{lemma}
	\begin{proof}
		Since time is fixed, we will often forget to write it in this proof.
		We observe first that, thanks to the assumption that the initial vorticity components are compactly supported, we know that $W_i(0)<\infty$. Because of the smoothness of $\omega_i$, the Wasserstein distance remains finite and differentiable for positive times.
		
		We consider the evolution of the squared Wasserstein distance, and compute, using the evolution  \eqref{3.5} of the vorticity components, multiple integrations by parts and the definition of the vorticity centers $X_i$,
		\begin{equation*}
			\begin{split}
				\frac{d}{dt}W_i^2 &= \frac{2}{a_i} \int_{\mathbb{R}^2} (x-X_i)\cdot u(x)\,\omega_i(x)\,dx + \frac{\nu}{a_i} \int_{\mathbb{R}^2} \Delta|x-X_i|^2 \,\omega_i(x)\,dx \\
				&\quad - \frac{2}{a_i} \int_{\mathbb{R}^2} (x-X_i)\, \omega_i(x)\,dx \cdot \frac{dX_i}{dt} \\
				&= \frac{2}{a_i} \iint_{\mathbb{R}^2\times \R^2} (x-X_i)\cdot K(x-y)\,\omega_i(x)\omega_i(y)\,dxdy \\
				&\quad+ \frac{2}{a_i} \int_{\mathbb{R}^2} (x-X_i)\cdot F_i(x)\,\omega_i(x)\,dx + 4\nu.
			\end{split}
		\end{equation*}
		The first integral on the right-hand side vanishes because $K$ is odd and because $z \cdot K(z) = 0$. Using the definition of $X_i$ again and decomposing $F_i=F_i^L+F_i^B$, we thus have
		\begin{align*}
			\frac{d}{dt}W_i^2(t) &= \frac{2}{a_i} \int_{\mathbb{R}^2} (x-X_i) \cdot \left(F_i^L(x)-F_i^L(X_i)\right) \omega_i(x)\,dx \\
			&\qquad + \frac{2}{a_i} \int_{\mathbb{R}^2} (x-X_i) \cdot F_i^B(x)\, \omega_i(x)\,dx + 4\nu.
		\end{align*}	
		The first integral is easy to estimate by $W_i^2/d^2$ because of the Lipschitz property of $F_i^L$.
		On the second integral we use the different bounds obtained for $F_i^B$. Splitting the integration domain into $B_{d/6}(X_i)$ and its complement, we have
		\begin{align*}
\MoveEqLeft
				\int_{\mathbb{R}^2} (x-X_i)
				 \cdot F_i^B(x)\, \omega_i(x)\,dx \\
				&\lesssim \|F_i^B\|_{L^\infty(B_{d/6}(X_i))} \int_{\mathbb{R}^2} |x-X_i|\,|\omega_i(x)|\,dx \\
				&\quad + \|F_i^B\|_{L^\infty(B_{d/6}(X_i)^c)} \left( \int_{B_{d/6}(X_i)^c} |\omega_i(x)|\,dx \right)^{\frac{1}{2}} \left( \int_{\mathbb{R}^2} |x-X_i|^2 \,|\omega_i(x)|\,dx \right)^{\frac{1}{2}} \\
				&\lesssim \|F_i^B\|_{L^\infty(B_{d/6}(X_i))} \, W_i +  \|F_i^B\|_{L^\infty(B_{d/6}(X_i)^c)} \left( \int_{B_{d/6}(X_i)^c} |\omega_i(x)|\,dx \right)^{\frac{1}{2}} \, W_i,
		\end{align*}
		where we used Jensen's and Hölder's inequalities.
		Observe that in the estimates above the assumption that the vortex components have definite sign was necessary to bound
		\begin{equation*}
			\int_{\mathbb{R}^2} |x-X_i|\,|\omega_i(x)|\,dx  \lesssim \left(\int_{\mathbb{R}^2} |x-X_i|^2 \,|\omega_i(x)|\,dx\right)^\frac{1}{2} \lesssim W_i.
		\end{equation*}
		It remains only to apply Lemma \ref{lem3.1}   and the proof is complete.
	\end{proof}

	Now we come to the objective of this subsection, that is a Gronwall estimate on the Wasserstein distance.
	Our argument will be the following. We start from the differential inequality \eqref{3.13}, and for small times we observe that the $L^p$ norm of the vorticity inherits the bound \eqref{3.30} from the initial datum. For larger times, the effect of the viscosity kicks in, and we use the decay estimate \eqref{3.14}. In both cases, in order to balance \eqref{3.30} or \eqref{3.14}, we need to make sure  that the portion of vorticity far from the centers is small enough.
	For this reason, we have to ensure not only that the vorticity centers remain sufficiently far from each other as in \eqref{3.52}, but also that most of each vorticity component remains concentrated around its center. We observe that, under the localization condition \eqref{3.2} and the defintion of $d $ in \eqref{305}, we know that 
	\begin{equation}\label{3.49}
		\min_{i\not=j}|\bar{X}_i-\bar{X}_j|\ge\frac34d ,
	\end{equation} 
	provided that $\eps\le d/8$, which we will suppose from here on. We select then $T=T(\eps,\nu,R)$ as the \emph{minimal} time for which at least one of the following equalities holds true:
	\[
	T = T_c
	\]
	\emph{or}
\[
 \max_i m_i(T,d/6)=\eps^\alpha+(\nu T)^{\frac{\alpha}{2}} ,
 \]
 \emph{or}
 \[
\max_i |X_i(T) - Y_i(T)| =\frac{d}{4} ,
 \]
	where $\alpha=\alpha(p,\gamma)< \beta$ is a positive parameter to be fixed later, cf.~\eqref{3.alpha}. Here we observe that, thanks to \eqref{3.2}, \eqref{3.49} and \eqref{303}, and because both the vortex centers $X_i$ and the outer vorticity $m_i$ are continuous in time,
	the time $T$ is always strictly positive, $T>0$, . We thus have
	\begin{equation}\label{3.t1}
	 m_i(t,d/6)\le\eps^\alpha+(\nu t)^{\frac{\alpha}{2}} \quad\text{and}\quad|X_i(t) - Y_i(t)| \le \frac{d}{4}
	\end{equation}
for all $t\le \min\{ T,T_c\}$ and any $i$. Moreover, using the triangle inequality and \eqref{305}, we also notice that the latter entails
\begin{equation}\label{311}
	|X_i(t)-X_j(t)|\ge\frac{d}{2}
\end{equation}
for all $t\le \min\{ T,T_c\}$ and any $i\not=j$.	We derive the estimate on the growth of $W_i$ for times smaller than $T$, while in Subsection \ref{subsec2.2} we will show that the time $T$ can be chosen independently of $\eps$ and~$\nu$.

	\begin{prop}\label{lem3.3}
		Suppose that  $\alpha$ satisfies
		\begin{equation}\label{3.alpha}
			\alpha \ge \frac1{\theta} +\left(\frac1{\theta}-1\right)\gamma  = \frac{3p-2}{p-2} + \frac{2p\gamma}{p-2},
		\end{equation}
		where $\theta$ was given as in Lemma \ref{lem3.1}. Then, for any $i=1,\ldots,N$ and $t\in[0,\min\{T,\nu^{-1}\}]$ there holds
		\begin{equation}\label{3.32}
			W_i^2(t) \lesssim  e^{Ct/d^2} (\varepsilon^2 + \nu t),
		\end{equation}
		for some positive constant $C>0$.
	\end{prop}
	\begin{proof}
		Using conditions \eqref{3.t1}, \eqref{305} and assumption \eqref{3.30} together with \eqref{3.4} in estimate \eqref{3.13}, we obtain, in the case $\sqrt{\nu t}\le\eps$,
		\begin{equation*}
			\frac{d}{dt}W_i^2(t) \lesssim d^{-2}W_i^2(t) + [d^{\theta}\varepsilon^{\theta\alpha - \gamma(1-\theta)} + d^{\frac{p-2}p}\varepsilon^{\frac{\alpha}{2}-\gamma} ] W_i(t) + \nu.
		\end{equation*}
		It is sufficient to assume that $\alpha$ is large enough, so that all the exponents of $\eps$ are larger or equal to $1$. This condition is guaranteed by \eqref{3.alpha} and hence, being that $\varepsilon < 1$, $R\lesssim 1$ and $\theta<\frac{p-2}p$, we have
		\begin{equation}\label{3.15}
			\frac{d}{dt}W_i^2(t) \lesssim d^{-2} W_i^2(t) + d^{\theta} \varepsilon W_i(t) + \nu\lesssim d^{-2} W_i^2(t) + d^{2\theta+2} \eps^2 +\nu,
		\end{equation}
		for these times. 
		There exists thus a constant $C>0$  such that
		\[
		\frac{d}{dt}\left( e^{-Ct/d^2}W_i^2(t)\right) \le Ce^{-Ct/d^2}\left(d^{2\theta+2}\eps^2+\nu\right),
		\]
		and an integration in time and using $d \sim 1$ yields
		\begin{equation}\label{200}
		W_i^2(t) \le e^{Ct/d^2} W_i^2(0) + \left( e^{Ct/d^2}-1\right)\left(d^{2\theta+4}\eps^2 +d^2 \nu\right) \lesssim e^{Ct/d^2} \left(\eps^2 +t \nu\right),
		\end{equation}
		for any $t\le \eps^2/\nu$, where we applied  assumption \eqref{3.2} and inequality \eqref{3.6}.
		
		For larger times, $\sqrt{\nu t}\ge\eps$, together with conditions \eqref{3.t1} and \eqref{305}, we use estimate \eqref{3.14} in the differential inequality \eqref{3.13}, and get
		\begin{equation*}
			\frac{d}{dt} W_i^2 \lesssim d^{-2} W_i^2 + d^{\theta}(\nu t)^{\frac{\alpha \theta}{2} - (1-\theta)\left(1-\frac1p\right)} W_i + d^{\frac{p-2}p}(\nu t)^{\frac{\alpha}4 -\left(1-\frac1p\right)} W_i +\nu.
		\end{equation*}
		Using condition \eqref{3.alpha} on $\alpha$ and the lower bound $\gamma\ge 2-2/p$  ensures that all the exponents of $\nu t$ are larger than $1/2$. Then thanks to the facts that $\nu t\le 1$ and $d\sim1$, we obtain the inequality 
		\begin{equation}\label{3.16}
			\frac{d}{dt}W_i^2(t) \lesssim d^{-2} W_i^2(t) + d^{\theta} \sqrt{\nu t} W_i(t) + \nu \lesssim d^{-2} W_i^2(t) + (d^{2\theta+2}t+1)\nu
		\end{equation}
		for these times. Similarly as above, we rewrite this differential inequality as
		\[
		\frac{d}{dt}\left(e^{-Ct/d^2}W_i^2(t)\right) \le C e^{-Ct/d^2}(d^{2\theta+2}t+1)\nu,
		\]
		with a (possibly larger) constant $C>0$. Integration in time over $t\ge \eps^2/\nu$ and using $d\sim 1 $ once more then yields
		\[
		W_i^2(t) \le e^{C(t-\eps^2/\nu)/d^2} W_i^2(\eps^2/\nu) +e^{C t/d^2} (\eps^2 +t\nu) \lesssim e^{Ct/d^2} \left(\eps^2 +t\nu\right),
		\]
		 where we have used the previous bound \eqref{200}. Combining both, we obtain our thesis. 
	\end{proof}
	
	From now on, we choose $\alpha$ such that \eqref{3.alpha} holds.
	Moreover, in what follows, we will write $\Lambda = C/d^2$ in the exponenential growth factor. 
	
	\subsection{Derivation of the main estimates}
So far, we have derived a bound on the spreading rate of the single components in terms of the second moment function $W_i(t)$, that is,  		
the Wasserstein distance between the vorticity components and their centers. However, we would like to express it in relation to the point-vortices solving \eqref{3.pvs}. Hence, in our next step we attempt to estimate the distance $|X_i-Y_i|$ between the vorticity centers and the point-vortices. 

	\begin{proof}[Proof of Theorem \ref{th3.2}]
		The proof is a simple adaptation of an argument in the proof of Theorem 2.2 from \cite{CetroneSerafini18}. Up to considering $\omega_j/a_j$ instead of $\omega_j$, we assume without loss of generality that $\omega_j$ is nonnegative and that $a_j=1$ for all $j$. Moreover, we often forget about the $t$'s.
		We compute
		\begin{equation*}
			\begin{split}
				\frac{dX_i}{dt}-\frac{dY_i}{dt} &= \sum_{j\neq i} \left[ \int_{\mathbb{R}^2}\int_{\mathbb{R}^2} K(x-y)\,\omega_i(x)\,\omega_j(y)\,dy\,dx - K(Y_i-Y_j) \right] \\
				&= A_1 + A_2 + A_3 + A_4,
			\end{split}
		\end{equation*}
		where
		\begin{align*}
			A_1 &=  \int_{\mathbb{R}^2} F_i^L(x) \omega_i(x)\,dx
			- F_i^L(Y_i),\\
			A_2 &= \int_{\mathbb{R}^2} F_i^B(x) \omega_i(x)\,dx,\\
			A_3 &= \sum_{j\neq i} \int_{\mathbb{R}^2} [K(Y_i-y)-K(Y_i-Y_j)](1-\psi(Y_i-y)) \,\omega_j(y)\,dy,\\
			A_4 & = -\sum_{j\neq i} \int_{\mathbb{R}^2} K(Y_i-Y_j)\,\psi(Y_i-y)\,\omega_j(y)\,dy,
		\end{align*}
		with  $\psi$ the same cutoff function that has been defined in \eqref{3.10}. Let us consider $A_1$ first. Recalling Proposition \ref{lem3.3},
		\begin{align*}
				A_1 &= \int_{\mathbb{R}^2} [F_i^L(x)-F_i^L(Y_i)]\,\omega_i(x)\,dx \\
				&\lesssim \frac1{d^2} \int_{\mathbb{R}^2} |x-X_i|\,\omega_i(x)\,dx + \frac1{d^2}|X_i-Y_i|
				\lesssim e^{\Lambda t} (\varepsilon+\sqrt{\nu t}) + |X_i-Y_i|,
		\end{align*}
		because $d\sim 1$. Concerning $A_2$, splitting the integration domain in $B_{d/6}(X_i)$ and its complement and using \eqref{305}, by Lemma \ref{lem3.1} we have
		\begin{equation*}
			\begin{split}
				|A_2| & \lesssim \|F_i^B\|_{L^\infty(B_{d/6}(X_i))} + \|F_i^B\|_{L^\infty(B_{d/6}(X_i))^c} \, m_i(t,d/6) \\
				&\lesssim   \sum_{j\neq i}m_j(t,d/6)^{\theta}\|\omega_j\|_{L^p}^{1-\theta} +  \sum_{j\neq i}\|\omega_j\|_{L^p}\,m_i(t,d/6). 
			\end{split}
		\end{equation*}
		We proceed similarly to the proof of Proposition \ref{lem3.3}.
In the case that $\sqrt{\nu t}<\eps$, we use \eqref{3.4}, \eqref{3.30}, \eqref{3.t1}, as well as condition \eqref{3.alpha} on $\alpha$ to obtain $|A_2|\lesssim \eps$ for these times. If $\sqrt{\nu t}\ge\eps$, we use estimates $\eqref{3.14}$, \eqref{3.t1} and condition \eqref{3.alpha} to obtain $|A_2|\lesssim \sqrt{\nu t}$ also in this case.
		We consider now $A_3$. 
		Because of the cutoff function $1-\psi$, we may restrict the integral to $|y-Y_i|>d/12$; this  ensures that $|K(Y_i-y) - K(Y_i-Y_j)|\lesssim |Y_j-y|/d^2 $ on the domain of integration. Hence we can bound 
		\begin{align*}
			|A_{3}| & \lesssim  \frac1{d^2} \sum_{j\neq i} \int_{\R^2} |y-Y_j|\, \omega_j(y)\, dy \lesssim e^{\Lambda t} (\varepsilon+\sqrt{\nu t}) + \sum_{j\neq i} |X_j-Y_j|
		\end{align*}
	by Proposition \ref{lem3.3}.
		We turn now to $A_4$. Since $|Y_i-Y_j|\ge d\gtrsim1$, the kernel $K(Y_i-Y_j)$ is bounded by $1$, and thus $|A_4| \lesssim \sum_{j\neq i}   \int_{B_{d/6}(Y_i)} \,\omega_j(y)\,dy$. We observe that on the domain of integration there holds
		\begin{equation*}
			|y-X_j| \ge |Y_j-Y_i|-|Y_i-y|-|Y_j-X_j| \ge \frac7{12}d,
		\end{equation*}
		again thanks to \eqref{3.t1}. Then $|A_4| \lesssim  \sum_{j\neq i} \int_{B_{d/6}(X_j)^c} \omega_j(y)\,dy = \sum_{j\neq i} m_j(t,d/6)$,  and again we use condition \eqref{3.t1} to bound $m_j$ with $\varepsilon$ or $\sqrt{\nu t}$ respectively in the case $\sqrt{\nu t}<\eps$ and $\sqrt{\nu t}\ge\eps$, keeping in mind that $\alpha>1$ by \eqref{3.alpha}.
		In conclusion, all the estimates obtained yield
		\begin{equation*}
			\left| \frac{d}{dt}|X_i-Y_i| \right| \le \left| \frac{dX_i}{dt}-\frac{dY_i}{dt} \right| \lesssim  \sum_j |X_j-Y_j| + e^{\Lambda t}(\varepsilon+\sqrt{\nu t})
		\end{equation*}
		for all $t\le T$. Summing over $i=1,\ldots,N$ and using a Gronwall argument, together with the initial condition that $|\bar{X}_i-\bar{Y}_i|\le\varepsilon$, yields the result. 
	\end{proof}

	\begin{proof}[Proof of Theorem \ref{th3.1}]
		This is an easy consequence of Proposition \ref{lem3.3} and Theorem \ref{th3.2} thanks to the triangle inequality.
	\end{proof} 
	
	\begin{proof}[Proof of Corollary \ref{cor3.1}]
		From the properties of the Wasserstein distance $W_1$ there follows
		\begin{equation*}
			\begin{split}
				W_1\left( \omega(t), \sum_i a_i \delta_{Y_i(t)} \right) \le \sum_i W_1(\omega_i(t), a_i\delta_{Y_i(t)}) = \sum_i |a_i| \int_{\mathbb{R}^2} |x-Y_i(t)|\, \frac{\omega_i(t,x)}{a_i} \,dx.
			\end{split}
		\end{equation*}
		The thesis is then a consequence of Jensen's inequality and Theorem \ref{th3.1}.
	\end{proof}

	\subsection{Proof that $T\ge \min\{T_c,c\log\frac1{\eps+\sqrt{\nu}}\}$.}\label{subsec2.2}

	In the following proofs we assume without loss of generality that $T\le T_c$ because otherwise there is nothing left to prove.
We distinguish two cases.
	
	\medskip
	
	In the \emph{first case}, we suppose that
		\begin{equation}\label{310}
		|X_i(T)-Y_i(T)| =  \frac{d}4
	\end{equation}
	for some $i$.

	\begin{lemma}\label{lem10}
		Suppose that \eqref{310} holds. Then $T\gtrsim \log\frac1{\eps + \sqrt{\nu}}$.
	\end{lemma}

	\begin{proof}
	This is any easy consequence of the first estimate in Theorem \ref{th3.2}. Indeed, by the virtue of \eqref{310}, it holds that
\[
 d\lesssim |X_i(T)-Y_i(T)| \lesssim e^{2\Lambda T}(\eps+\sqrt{\nu}),
\]
which yields the statement of the lemma upon taking the logarithm and choosing $\eps+{\sqrt{\nu}}$ sufficiently small.
	\end{proof}
	
	\medskip

	\medskip
	
	In the \emph{second case}, we assume that for some $i\in\{1,\ldots,N\}$
	\begin{equation}\label{3.45}
		m_i(T,d/6) = \eps^\alpha+(\nu T)^{\frac{\alpha}{2}}.
	\end{equation}
	We endeavour to prove that $T$ is bounded uniformly from below in $\eps$ and $\nu$ also in this case. In order to do this, we derive an estimate on the outer vorticity portion $m_i(T,d/6)$ by employing an iteration estimate  similar to the one used in \cite{CeciSeis21b}. This in turn is based on the iterative procedure developed by Marchioro and co-workers, see for example \cite{Marchioro98_2, CapriniMarchioro15} but, differently from these works, we start from the assumption \eqref{303} that the initial vortex components are quickly decaying outside of balls of  radius $R$, while \cite{Marchioro98_2, CapriniMarchioro15} make use to a great extent of their hypothesis that this radius is of the order of $\eps$. 
	
	To estimate $m_i(T,d/6)$, we introduce the \enquote{smoothened outer vorticity portion} $\mu_i(t,\rho,r)$ defined in the following. We consider a smooth radially symmetric cut-off function $\eta=\eta_{\rho,r}$ such that
	\begin{equation*}
		\eta(x)=1 \text{ if } |x|\le\rho, \quad \eta(x)=0 \text{ if } |x|>\rho+r,
	\end{equation*}
	with
	\begin{equation}\label{3.36}
		|\nabla\eta| \lesssim \frac{1}{r}, \quad |\nabla^2\eta|\lesssim\frac{1}{r^2}.
	\end{equation}
	The $i$-th smoothened outer vorticity portion is defined as
	\begin{equation*}
		\mu_i(t,\rho,r) = \frac{1}{a_i} \int_{\R^2} \big( 1-\eta_{\rho,r}(x-X_i(t)) \big) \omega_i(t,x)\,dx.
	\end{equation*}
	It is easy to check that the functions $m_i$ and $\mu_i$ satisfy the relation
	\begin{equation}\label{3.51}
		m_i(t,\rho+r) \le \mu_i(t,\rho,r) \le m_i(t,\rho).
	\end{equation}
	We show a differential inequality for $\mu_i$ by making use of the fact that, thanks to Proposition \ref{lem3.3},
	\begin{equation}\label{3.35}
		m_i(t,\rho) \le \frac{1}{\rho^2} W_i(t)^2 \lesssim \frac{ e^{\Lambda t}}{\rho^2}(\eps^2+\nu t) \lesssim \frac{e^{\Lambda t}}{\rho^2}(\eps^2+\nu t)
	\end{equation}
	holds for all $t\le T\le 1$.
	
	\begin{lemma}\label{lem3.6}
		Let $i\in\{1,\ldots,N\}$ be fixed. For any $t\le T$, $r>0$ and $\rho\in(R,d/6-r)$ it holds that
		\begin{equation}\label{3.38}
			\mu_i(t,\rho,r) \le \mu_i(0,\rho,r) + \kappa(\rho,r) \int_{0}^{t} \mu_i(s,\rho-r,r)\,ds,
		\end{equation}
		where
		\begin{equation}\label{3.39}
			\kappa(\rho,r) \lesssim \frac{e^{2\Lambda T}}{\rho^2}  \left(\frac{\eps+\sqrt{\nu}}r + \frac{\eps^2+\nu}{r^2}\right)+\frac{\rho}r
		\end{equation}
	\end{lemma}
	\begin{proof}
		Up to dividing by $a_i$, we can suppose without loss of generality that $\omega_i\ge0$ and $a_i=1$. Moreover, we often neglect the time $t$  and we write $\mu_i(t,\rho) = \mu_i(t,\rho,r)$ for notational convenience.
		
		Because $\omega_i$ solves the advection-diffusion equation \eqref{3.5}, and thanks to the fact that
		\begin{equation}\label{3.18}
			\frac{d}{dt}X_i(t) = \frac{1}{a_i} \int_{\mathbb{R}^2} u(t,x)\omega_i(t,x)\,dx = \frac{1}{a_i} \int_{\mathbb{R}^2} F_i(t,x)\omega_i(t,x)\,dx,
		\end{equation}
		the second identity being true by the fact that the self-interaction term vanishes because of the property $K(z)=-K(-z)$, we can compute
		\begin{equation*}
			\begin{split}
				\frac{d}{dt}\mu_i(t,\rho) &= \frac{d}{dt}X_i \cdot \int_{\R^2} \nabla\eta(x-X_i)\,\omega_i(x)\,dx \\
				&\qquad- \int_{\R^2} (u_i(x)+F_i(x))\cdot\nabla\eta(x-X_i)\,\omega_i(x)\,dx \\
				&\qquad\qquad -\nu \int_{\R^2} \Delta \eta(x-X_i)\, \omega_i(x)\,dx \\
				&= -\int_{\R^2}\int_{\R^2} K(x-y)\cdot \nabla\eta(x-X_i)\, \omega_i(y)\,\omega_i(x)\,dy\,dx \\
				&\qquad + \int_{\R^2}\int_{\R^2} [F_i(y)-F_i(x)]\cdot\nabla\eta(x-X_i)\,\omega_i(y)\,\omega_i(x)\,dy\,dx \\
				&\qquad\quad -\nu \int_{\R^2} \Delta \eta(x-X_i)\, \omega_i(x)\,dx \\
				&= I_1 + I_2 + I_3.
			\end{split}
		\end{equation*}
Using once more $K(z)=-K(-z)$ we may write
		\begin{equation*}
			I_1 = \frac12 \int_{\R^2}\int_{\R^2} (\nabla \eta(y-X_i)-\nabla\eta(x-X_i)) \cdot K(x-y)\,\omega_i(x)\,\omega_i(y)\,dx\,dy.
		\end{equation*}
		We notice that the integrand in non-zero only if either $\rho<|x-X_i|<\rho+r$ or $\rho<|y-X_i|<\rho+r$.
		We can therefore split the integration domain into the sets
		\begin{align*}
			A = \left(B_{\rho+r}(X_i)\setminus B_{\rho}(X_i)\right) \times B_{\frac{\rho}2}(X_i),\quad
			B = B_{\frac{\rho}2}(X_i) \times \left(B_{\rho+r}(X_i)\setminus B_{\rho}(X_i)\right),\\
			C  = ((B_{\rho+r}(X_i)\setminus B_{\rho}(X_i))\times B_{\frac{\rho}2}^c(X_i))\cup (B_{\frac{\rho}2}^c(X_i)\times (B_{\rho+r}(X_i)\setminus B_{\rho}(X_i))),
		\end{align*}
		and we denote by $I_1^A$, $I_1^B$ and $I_1^C$ the contributions to $I_1$ due to the sets $A$, $B$ and $C$ respectively. The terms $I_1^A$ and $I_1^B$ can be estimated in almost the same way. Using the radial symmetry of the cut-off function and orthogolality, we notice that
		\[
		\grad\eta(x-X_i)\cdot K(x-y)  = \eta'(|x-X_i|) \frac{x-X_i}{|x-X_i|}\cdot K(x-y) = \eta'(|x-X_i|) \frac{y-X_i}{|x-X_i|}\cdot K(x-y),
		\]
		and thus, since on their domain of integration it holds that $|x-y|\ge \rho/2$, using the scaling of the gradient of  $\eta$ from \eqref{3.36} and Proposition \ref{lem3.3} , we see that
		\begin{align*}
			|I_1^A| &   = \left|\frac{1}{2}\iint_A |\eta'(|x-X_i|) | \frac{|y-X_i|}{|x-X_i|}| K(x-y)|\omega_i(x)\omega_i(y)\, dxdy\right| \\
			&\lesssim \frac1{r\rho^2}     m_i(t,\rho)W_i(t)\\
			&\lesssim \frac{e^{\Lambda t}}{\rho^2} \frac{\eps+\sqrt{\nu t}}r m_i(t,\rho),
		\end{align*}
and the estimate of $I_1^B$ proceeds analogously. To estimate $I_1^C$ we use the Lipschitz condition \eqref{3.36} and estimate \eqref{3.35} to obtain
		\begin{align*}
			|I_1^C| &\lesssim	\frac{1}{r^2} \iint_C |x-y||K(x-y)| \omega_i(x)\omega_i(y)\, dxdy \\
			&\lesssim \frac{1}{r^2} m_i\left(t,\frac{\rho}{2}\right) m_i(t,\rho)  \\
			& \lesssim \frac{e^{\Lambda t}}{\rho^2}\frac{\eps^2+\nu t}{r^2}\,m_i(t,\rho).
		\end{align*}
		
		Now we consider $I_2$. We split $F_i = F_i^L + F_i^B$,
		\begin{equation*}
			\begin{split}
				I_2 &= \int_{\R^2}\int_{\R^2} [F^L_i(y)-F^L_i(x)]\cdot\nabla\eta(x-X_i)\,\omega_i(y)\,\omega_i(x)\,dydx \\
				& \quad + \int_{\R^2}\int_{\R^2} [F^B_i(y)-F^B_i(x)]\cdot\nabla\eta(x-X_i)\,\omega_i(y)\,\omega_i(x)\,dydx \\
				& = I_{21} + I_{22}.
			\end{split}	
		\end{equation*}
		We use the Lipschitz-continuity of $F_i^L$ given by Lemma \ref{lem3.1}, property \eqref{3.36}, the triangle and Jensen's inequalities and Lemma \ref{lem3.3} to write
		\begin{equation*}
			\begin{split}
				|I_{21}| &\lesssim \frac{1}{r} \int_{\R^2} \int_{\rho<|x-X_i|<\rho+r} |x - y| \, \omega_i(x)\,\omega_i(y)\,dx\,dy \\
				& \lesssim \frac{1}{r} \int_{\rho<|x-X_i|<\rho+r} |x - X_i| \, \omega_i(x) \,dx \\
				& \quad + \frac{1}{r} \left( \int_{|x - X_i|> \rho} \omega_i(x)\,dx \right)  \left( \int_{\R^2} |y - X_i| \,\omega_i(y)\,dy \right) \\
				& \lesssim \left(1+\frac{\rho}{r}\right)\, m_i(t,\rho) + e^{\Lambda t} \frac{\eps+\sqrt{\nu t}}{r}\, m_i(t,\rho).
			\end{split}
		\end{equation*}
Now we consider $I_{22}$. By our choice of $r$ and $\rho$, it holds that $\rho + r \le d/6$, and thus the integrand vanishes when $x\notin B_{d/6}(X_i)$. We can therefore split the integration domain of $I_{22}$ into the following sets,
		\begin{equation*}
			D = B_{d/6}(X_i) \times B_{d/6}(X_i), \quad  E = B_{d/6}(X_i) \times B_{d/6}(X_i)^c,
		\end{equation*}
		and investigate the two contributions $I_{22}^D$ and $I_{22}^E$ separately. On the set $D$, we use \eqref{3.36} so that
		\begin{equation*}
			|I_{22}^D| \lesssim \|F_i^B\|_{L^\infty(B_{d/6}(X_i))} \frac{1}{r} m_i(t,\rho).
		\end{equation*}
Thanks to Lemma \ref{lem3.1}  and property \eqref{3.t1}, it holds that
		\begin{equation*}
				\|F_i^B\|_{L^\infty(B_{d/6}(X_i))} \lesssim  (\eps^\alpha+(\nu t)^{\frac{\alpha}{2}})^\theta \sum_{j\not=i}\|\omega_j\|_{L^p}^{1-\theta},
		\end{equation*}
		and then, by the non-increasing property \eqref{3.4} of the $L^p$ norm, the scaling assumption \eqref{3.30}, condition \eqref{3.alpha} on $\alpha$ and the bound \eqref{3.14} on $\omega_i$, we find
		\begin{equation}\label{3.37}
			\|F_i^B\|_{L^\infty(B_{d/6}(X_i))} \lesssim \begin{cases}
				 \eps^{\alpha\theta-(1-\theta)\gamma} <  \eps &\text{if } \sqrt{\nu t} \le \eps \\
				 (\nu t)^{\frac{\alpha \theta}{2}-(1-\theta)\frac{p-1}{p}} <  \sqrt{\nu t} &\text{if } \sqrt{\nu t} > \eps,
			\end{cases}
		\end{equation}
		also because $\nu t< 1$.
		Hence
		\begin{equation*}
			|I_{22}^D| \lesssim \frac{\eps+\sqrt{\nu t}}{r} m_i(t,\rho).
		\end{equation*}
		We estimate the contribution $I_{22}^E$ due to the set $E$ as
		\begin{equation*}
			\begin{split}
				|I_{22}^E|&\lesssim \|F_i^B\|_{L^{\infty}} \frac1{r} m_i(t,\rho)m_i(t,d/6)\\
				&\lesssim \|F_i^B\|_{L^\infty(B_{d/6}(X_i)^c)} \, m_i(t,d/6)\, \frac{1}{r} \,m_i(t,\rho) + \|F_i^B\|_{L^\infty(B_{d/6}(X_i))} \frac{1}{r} \,m_i(t,\rho).
			\end{split}
		\end{equation*}
		Now, thanks to the first inequality in \eqref{3.t1} and Lemma \ref{3.1}, we have
		\begin{equation*}
			\|F_i^B\|_{L^\infty(B_{d/6}(X_i)^c)} \, m_i(t,d/6) \lesssim  \|\omega(t)\|_{L^p} \, (\eps^\alpha+(\nu t)^{\frac{\alpha}{2}}).
		\end{equation*}
		Invoking property \eqref{3.4}, assumption \eqref{3.30}, the heat-kernel-type bound \eqref{3.14}, and condition \eqref{3.alpha} on $\alpha$, we may hence write	
		\begin{equation*}
				\|F_i^B\|_{L^\infty(B_{d/6}(X_i)^c)} \, m_i(t,d/6) \lesssim \begin{cases}
					  \eps^{\alpha-\gamma} <   \eps &\quad\text{if } \sqrt{\nu t} \le \eps \\
					  (\nu t)^{\frac{\alpha}{2}-1+\frac{1}{p}} <   \sqrt{\nu t} &\quad\text{if } \sqrt{\nu t} > \eps.
				\end{cases}
		\end{equation*}
		From this and \eqref{3.37}, we see that
		\begin{equation*}
			|I_{22}^E|\lesssim \frac{ \eps+\sqrt{\nu t}}r \,m_i(t,\rho).
		\end{equation*}

		The term $I_{3}$ can be easily estimated using the bound \eqref{3.36},
		\begin{equation*}
			|I_3| \lesssim \frac{\nu}{r^2} m_i(t,\rho).
		\end{equation*}
		
		In conclusion, thanks to relation \eqref{3.51} and $r\le \rho\lesssim  1$, 
		\begin{align*}
			\left| \frac{d}{dt} \mu_i(t,\rho) \right| &\lesssim \frac{e^{2\Lambda T}}{\rho^2}  \left(\frac{\eps+\sqrt{\nu}}r + \frac{\eps^2+\nu}{r^2}\right)\mu_i(t,\rho-r)+\frac{\rho}r \mu_i(t,\rho-r).
		\end{align*}
		Integrating in time, we obtain the thesis.
	\end{proof}
	
	Iterating estimate \eqref{3.38}, we come to the lower bound of $T$.
	
	\begin{prop}\label{lem3.7}
Suppose that \eqref{3.45} holds and $\beta\ge 3\alpha$. Then there exist $\eps_0$, $\nu_0\in(0,1)$ and a constant $c$ such that for any $\eps\le \eps_0$ ad $\nu\le \nu_0$ it holds that $T\ge c$. If $T_c=\infty$ and $\delta>0$ in \eqref{317}, there is the stronger estimate
\[
T\ge c\log\frac1{\eps+\sqrt{\nu}}.
\]
	\end{prop}
	\begin{proof}We fix $i$ such that \eqref{3.45} holds.
		Without loss of generality, we suppose that $a_i=1$ and $\omega_i\ge 0$. Moreover, we assume that
		\begin{equation}
		\label{313}
		T\le \frac{1-2\delta}{4\Lambda}\log \frac{1}{\eps+\sqrt{\nu}},
		\end{equation}
		because otherwise, there is nothing left to prove.
		
		Let $M\in\N$ be a natural number. Our goal  to decrease the radii in $M$ steps from $d/6$ to $R$, so that $\rho_m = d/6-\sum_{n=1}^{m}r_n$ partitions $[R,d/6]$ for some $r_n\in(0,R)$. We want to choose $r_m$ so that
		\begin{equation}\label{315}
		\eps+\sqrt{\nu}\le r_m^{\frac{2}{1-2\delta}},
\end{equation}	
which implies, also thanks to \eqref{313}, that the constant in \eqref{3.39} is further estimated by
\begin{equation}
\kappa \lesssim \frac{(\eps+\sqrt{\nu})^{2\delta}}{\rho_m^2} +\frac{\rho_m}{r_m} .
\label{316}
\end{equation}
To specify our choice of $r_m$, we make the ansatz
\[
r_m = \frac{\sigma}M\left(\log\frac1{\eps+\sqrt{\nu}}\right)^{\xi} \rho_m ,
\]
for some $\sigma$ and $\xi$ in the interval $(0,1)$, which we will both fix later. Solving the previous identity for $r_m$, we find the iterative formula
\[
r_m = \chi\left(\frac{d}6-\sum_{n=1}^{m-1}r_n\right),\quad \chi =\frac{\frac{\sigma}M\left(\log\frac1{\eps+\sqrt{\nu}}\right)^{\xi}}{1+\frac{\sigma}M\left(\log\frac1{\eps+\sqrt{\nu}}\right)^{\xi}}\in(0,1),
\]
from which we deduce that
\[
r_m = \chi (1-\chi)^{m-1}\frac{d}6.
\]
It is readily verified that the radii $r_m$ are decreasing, $r_{m+1}< r_m$. We then compute that
\[
\rho_M= \frac{d}6-\sum_{m=1}^Mr_m  \ge \left(1-\chi\sum_{m=0}^{M-1}(1-\chi)^m\right)\frac{d}6 =(1-\chi)^M\frac{d}6 ,
\]
and the expression on the right-hand side is larger than $R$ only if
\begin{equation}
\label{319}
M\le \frac{\log\frac{d}{6R}}{\log\frac1{1-\chi}},
\end{equation}
which we suppose from here on. It follows via \eqref{317} that the first term in \eqref{316} is bounded by $1$, while the second term is larger than $1$ provided that
\begin{equation}
\label{318}
\sigma\left(\log\frac1{\eps+\sqrt{\nu}}\right)^{\xi} \le M.
\end{equation}
If $M$ can be chosen this way, which we will verify later, it follows that
\begin{equation}
\label{320}
\kappa(\rho_m,r_m) \lesssim M \left(\log\frac{1}{\eps+\sqrt{\nu}}\right)^{-\xi}.
\end{equation}

We now use relation \eqref{3.51} and apply Lemma \ref{lem3.6} to estimate
\begin{align*}
m_i(T,d/6) \le \mu_i(T,\rho_1,r_1) \le m_i(0,R) +\kappa(\rho_1,r_1)\int_0^T \mu_i(t_1,\rho_2,r_1)\, dt_1.
\end{align*}
The outer vorticity portions are monotone in $r$ in the sense that $\mu_i(t,\rho_m,r_{m-1})\le \mu_i(t,\rho_m,r_m)$, and thus, via an interation
		\begin{align*}
				\MoveEqLeft m_i(T,d/6) \\
				&\le  \sum_{m=0}^{M-1} \frac1{m!} \left(\prod_{n=1}^m\kappa(\rho_n,r_n)\right)  T^m m_i(0,R)\\
				&\quad + \left(\prod_{m=1}^M \kappa(\rho_m,r_m)\right) \int_{0}^{T}\int_{0}^{t_1}\ldots\int_{0}^{t_{M-1}}\mu_i(t_M,\rho_{M+1},r_M)\,dt_M\ldots dt_2 dt_1.
		\end{align*}
Using the trivial estimate $\mu_i(t,\rho,r)\le 1$, which holds true for all $t$, $\rho$ and $r$,  invoking the decay assumption of the initial configuration in \eqref{303}, the defining condition on $T$ in \eqref{3.45} and the bound in \eqref{320}, we obtain
		\begin{equation*}
			\eps^{\alpha} + (\nu T)^{\frac{\alpha}2} = m_i(T,d/6) \le \sum_{m=0}^{M-1} \frac1{m!} \left(\frac{C M  T }{\log^{\xi}\frac1{\eps+\sqrt{\nu}}}\right)^m   \eps^{\beta}+ \frac{1}{M!}\left(\frac{CMT}{\log^{\xi}\frac1{\eps+\sqrt{\nu}}}\right)^M,
		\end{equation*}
for some universal constant $C$. Using the Stirling formula $M^M < e^M M!$ and its generalization \eqref{300} in the appendix, we find that
\begin{equation}\label{102}
\begin{aligned}
\eps^{\alpha} + (\nu T)^{\frac{\alpha}2}& \le   \left(1+ \frac{CT}{\log^{\xi} \frac1{\eps+\sqrt{\nu}}}\right)^M\eps^{\beta} + \left(\frac{CT}{\log^{\xi}\frac1{\eps+\sqrt{\nu}}}\right)^M\\
&\le 2\left(\eps^{\frac{\beta}M} + \frac{CT}{\log^{\xi}\frac1{\eps+\sqrt{\nu}}}\right)^M,
\end{aligned}
\end{equation}
for some  new constant  $ C $. From here, we derive the desired bound by distinguishing two cases. 

\emph{Relatively small viscosity.} We deduce from \eqref{102} that
\begin{equation}
\label{322}
\eps^{\frac{\alpha}{M}} \le 2^{\frac1M}\left(\eps^{\frac{\beta}M} + \frac{CT}{\log^{\xi}\frac1{\eps+\sqrt{\nu}}}\right),
\end{equation}
and the left-hand side is uniformly bounded from below, say $\eps^{\frac{\alpha}{M}}= 1/2$ provided that 
\begin{equation}
\label{321}
\frac{\alpha }{\log 2}\log \frac1{\eps} = M .
\end{equation}
Before continuing, we want to make sure that such an $M$ is consistent with all the hypotheses we made so far, i.e., \eqref{315}, \eqref{319} and \eqref{318}. Let us start by  treating the case
\[
\eps\ge \sqrt{\nu}.
\]
First, regarding \eqref{319}, we notice that
\[
\frac1{1-\chi} = 1+\frac{\sigma}M \log^{\xi}\frac1{\eps+\sqrt{\nu}} \le 1+\frac{\sigma}{M} \log^{\xi}\frac1{\eps}  = 1+ \frac{\sigma\log 2}{\alpha}\frac1{\log^{1-\xi}\frac1{\eps}}.
\]
Hence, in order to guarantee \eqref{319}, it is enough to require that
\[
\frac{\alpha}{\log 2}  \log \left(1+ \frac{\sigma\log 2}{\alpha}\frac1{\log^{1-\xi}\frac1{\eps}}\right) \log \frac1{\eps}\le \log \frac{d}{6R}.
\]
The latter can be strengthend by estimating the logarithm linearly,
\[
\sigma \log^{\xi}\frac1{\eps} \le \log \frac{d}{6R}.
\]
If now \eqref{317} holds true with $\delta=0$, that is $R\gtrsim1$, we necessarily have to choose $\xi=0$  and this estimate holds true for $\sigma$ sufficiently small. Otherwise, if $\delta >0$, we are allowed to choose $\xi=1$ and we have to ensure that
\[
\sigma \log \frac1{\eps} \le \delta \log\frac1{\eps} -c,
\]
for some universal constant $c\ge 0$. This is possible whenever $\sigma\le \delta$ and $\eps$ sufficiently small.

We now turn to condition \eqref{318}, which can be rewritten as
\[
\sigma \log^{\xi} \frac1{\eps+\sqrt{\nu}} \le \frac{\alpha}{\log 2}\log\frac1{\eps},
\]
which can be strengthend if $\sqrt{\nu}$ on the left-hand side is dropped and $\xi $ is estimated by $1$. In this case, choosing $\sigma$ small does the job.

We finally turn to the worst case of \eqref{315}, which can be rewritten as
\[
(M-1)\log \frac1{1-\chi} + \log\frac1{\chi} \le \frac{1-2\delta}2 \log\frac1{\eps+\sqrt{\nu}} +\log\frac{d}6.
\]
Estimating $\eps\le \eps+\sqrt{\nu}\le 2\eps$ and arguing similarly as above, we notice that this estimate can be strengthend to
\[
M \frac{\sigma\log 2}{\alpha} \frac1{\log^{1-\xi}\frac1{\eps}} + \log\left(1+ \frac{M}{\sigma} \log^{-\xi} \frac1{2\eps}\right) \le \frac{1-2\delta}2\log \frac1{\eps} + \log \frac{d}{12}.
\]
In view of our defintion of $M$, this estimate can be rewritten as
\[
\sigma  \log^{\xi}\frac1{\eps}  + \log\left(1+ \frac{\alpha}{\sigma\log 2 } \log^{1-\xi} \frac1{2\eps}\right) \le \frac{1-2\delta}2\log \frac1{\eps} + \log \frac{d}{12},
\]
which holds for $\xi\in\{0,1\}$ true for any fixed $\sigma<\frac{1-2\delta}2$ and $\eps$ small enough.

Reviewing the previous arguments, it is easy to check that in the particular case  $\xi=\delta=0$,  the choice of $M$ is possible even if 
\begin{equation}\label{324}
\eps  \le \sqrt{\nu} \le \frac{c}{\log^2 \frac1{\eps}} ,
\end{equation}
for some constant $c>0$. Indeed, \eqref{319} and \eqref{318} are simply equivalent to $\sigma \le \log\frac{d}{6R}$ and $\sigma\le M$, while for \eqref{315}, it is enough to require that
\[
\log \left(1+\frac{M}{\sigma}\right) + M \log\left(1+\frac{\sigma}{M}\right)\le \frac12\log \frac1{\eps +\sqrt{\nu}} + \log \frac{d}{6}.
\]
Estimating the second logarithm on the left-hand side linearly, and supposing that $2 \sigma \le \log \frac{d}6$, the latter can be strengthend to
\[
1+ \frac{M}{\sigma} \le \left(\frac{d/6}{\eps+\sqrt{\nu}}\right)^{1/2}.
\] 
Using our choice of $M$ in \eqref{321} and solving for $\nu$ gives \eqref{324}.

Now that we have proved that \eqref{321} is admissible, we want to pick $\beta$ such that
\[
\beta \ge \alpha + (M+1)\frac{\log2}{\log\frac1\eps}.
\]
This choice allows us to deduce from \eqref{322} that
\[
2^{M+1}\eps^{\beta} \le \eps^{\alpha}\le 2^{M+1}\left(\frac{C T}{\log^{\xi}\frac1{\eps+\sqrt{\nu}}}\right)^M,
\]
which in turn implies that
\[
\log^{\xi}\frac1{\eps+\sqrt{\nu}} \lesssim T,
\]
where $\xi=0$ if $\delta=0$ and $\xi =1$ if $\delta>0$ in \eqref{317}. It remains to notice that $\beta=3\alpha$ is possible if $\eps$ is small.

\medskip

\emph{Relatively large viscosities.} Let us start with the particular situation where $\xi=\delta=0$ and  
\begin{equation}\label{325}
\sqrt\nu \ge \frac{c}{\log^2\frac1{\eps}} .
\end{equation}
We start again with \eqref{102}, from which we derive this time the estimate
\[
\nu T \le  2^{\frac2{\alpha}} \left(\eps^{\frac{\beta}M} +CT \right)^{\frac{2M}{\alpha}},
\]
and we choose $\beta$ such that  $\eps^{\frac{\beta}M}\le CT$, that is
\begin{equation}
\label{104}
M\log\frac{1}{CT} \le \beta\log\frac1{\eps}.
\end{equation}
Then we obtain
\[
\nu^{\frac{\alpha}{2M-\alpha}} \le 2^{\frac{2M+2}{2M-\alpha}} C^{\frac{2M}{2M-\alpha}} T\le 4C^2T,
\]
and the last inequality is true as long as $M\ge \alpha+1$. The left-hand side is uniformly bounded from below, say $\nu^{\frac{\alpha}{2M-\alpha}}=1/2$, if
\begin{equation}\label{334}
\frac{\alpha}2 \log\frac2{\nu}= M.
\end{equation}
In this case, we thus have that 
\begin{equation}\label{335}
T\ge 1/ (8C^2).
\end{equation}
 Notice that under the above estimate, the condition $M\ge \alpha+1$ is automatically satisfied if $\nu$ is sufficiently small. 
 
 We have to check if the choice \eqref{334} of $M$ aligns with the assumption \eqref{315}, recalling that \eqref{319} and \eqref{318} are readily verified for $\xi=0$ if $\sigma$ is sufficiently small. We rewrite \eqref{315} with $m=M$ as
 \[
 \left(\eps+\sqrt{\nu} \right)^{\frac{1-2\delta}2}\le \chi(1-\chi)^{M-1}\frac{d}6 =\frac{\sigma}{M+\sigma} \left(\frac{M}{M+\sigma}\right)^{M-1} \frac{d}6,
 \]
 Using $\eps\le \sqrt{\nu}$ and the linear estimate for the logarithm, we observe that this estimate follows from
 \[
 \sigma + \log\left(1+\frac{\alpha}{2 \sigma} \log\frac2{\nu}\right) \le \frac{1-2\delta}{2}  \log\frac1{\sqrt{\nu}} + \log\frac{d}{12},
 \]
 which holds true for $\nu$ small enough.
 
Having established \eqref{335},  we finally claim that \eqref{104} is satisfied with $\beta=3\alpha$, i.e.,
 \[
 M\log\frac1{CT} \le 3\alpha \log\frac1{\eps}.
 \]
 Using the definition of $M$ in \eqref{334} and the lower bound on $T$ in \eqref{335}, we observe that the latter holds true if
 \[
 \log\frac{2}{\nu}\log (8C) \le 6\log\frac1{\eps}.
 \]
 Fortunately, we have supposed an exponential bound on $\eps$ via \eqref{325}, so that the previous estimate is implied by
 \[
 \log \frac2{\nu} \log(8C) \le 6\left( \frac{ c}{\sqrt{\nu}}\right)^{1/2}.
 \]
 This, in turn, is true whenever $\nu$ is small enough.
 
 The last situation to consider is the case 
 \[
 \eps\le \sqrt{\nu}
 \]
 and $\delta>0$, so that $\xi=1$. Since the case $\delta=0$ is weaker then the case $\delta>0$, cf.~\eqref{303} and \eqref{317}, the bound on $T$ in \eqref{335} carries over  to the case $\delta>0$. Estimate \eqref{102} thus becomes
 \[
 \left(\frac{\nu}{8C^2}\right)^{\frac{\alpha}{2M}} \le 2^{\frac1M} \left(\nu^{\frac{\beta}{2M}} + \frac{CT}{\log\frac1{\eps+\sqrt{\nu}}}\right),
 \]
 and a lower bound of the form
 \[
 T\gtrsim \log\frac1{\eps+\sqrt{\nu}}
 \]
 can be established almost identically to the case $\eps\ge \sqrt{\nu}$.	\end{proof}

	\section*{Appendix: An elmentary estimate}

In this appendix, we prove the elementary auxiliary estimate
\begin{equation}
\label{300}
\sum_{m=0}^M \frac1{m!} (sM)^m \le (1+es)^M,
\end{equation}	
	for any $m\in \N$ and $s>0$. In fact, thanks to the binomial theorem, it is enough to establish
	\begin{equation}
	\label{301}
	M^m \le \frac{M!}{M-m!} e^m,
	\end{equation}
	for any $M\in \N$ and $m\in \{0,1,\dots,M\}$. For $m=M$, this  just follows from  Stirling's formula. Moreover, the inequality for $m \in\{0,1,2\}$ is readily verified. 
	
	In order to establish the general estimate,   we rewrite \eqref{301} for $m\ge 2$ as
	\begin{equation}\label{302}
	\frac1{e} \le F(m):= \prod_{k=1}^{m-1}f(k),\quad f(k) = e\left(1-\frac{k}M\right),
	\end{equation}
	and notice that $f(k)\ge 1$ precisely if $k\le k_*:=M\left(1-\frac1e\right)$. Therefore, the mapping $m\mapsto F(m)$ is increasing on $ [2,k_*+1]$ and decreasing on $[k_*+1,M]$. It follows that $F(m) \ge \min\{F(2),F(M)\}$. We deduce \eqref{302} (and thus \eqref{301}) because we had showed that  \eqref{301} (and thus \eqref{302}) hold true for $m=2$ and $m=M$.

	\section*{Acknowledgments}
	
	This work is funded by the Deutsche Forschungsgemeinschaft (DFG, German Research Foundation) under Germany's Excellence Strategy EXC 2044--390685587, Mathematics M\"unster: Dynamics Geometry Structure.

	\bibliography{euler}

\begin{thebibliography}{10}

\bibitem{BenArtzi94}
{\sc Ben-Artzi, M.}
\newblock Global solutions of two-dimensional {N}avier-{S}tokes and {E}uler
  equations.
\newblock {\em Archive for Rational Mechanics and Analysis}, 128 (1994),
  329--358.

\bibitem{BenedettoCagliotiMarchioro00}
{\sc Benedetto, D., Caglioti, E., and Marchioro, C.}
\newblock {On the motion of a vortex ring with a sharply concentrated
  vorticity}.
\newblock {\em Math. Methods Appl. Sci. 23}, 2 (2000), 147--168.

\bibitem{ButtaMarchioro18}
{\sc Butt\`a, P., and Marchioro, C.}
\newblock Long time evolution of concentrated {E}uler flows with planar
  symmetry.
\newblock {\em SIAM J. Math. Anal. 50}, 1 (2018), 735--760.

\bibitem{ButtaMarchioro19}
{\sc Butt\`a, P., and Marchioro, C.}
\newblock Time evolution of concentrated vortex rings.
\newblock {\em J. Math. Fluid. Mech. 22}, 19 (2020).

\bibitem{ButtaCavallaroMarchioro21}
{\sc Buttà, P., Cavallaro, G., and Marchioro, C.}
\newblock Global time evolution of concentrated vortex rings.
\newblock {\em Preprint arXiv:2102.07807v2\/} (2021).

\bibitem{CapriniMarchioro15}
{\sc Caprini, L., and Marchioro, C.}
\newblock Concentrated {E}uler flows and point vortex model.
\newblock {\em Rend. Mat. Appl. (7) 36}, 1-2 (2015), 11--25.

\bibitem{CeciSeis21}
{\sc Ceci, S., and Seis, C.}
\newblock Vortex dynamics for 2{D} {E}uler flows with unbounded vorticity.
\newblock {\em Rev. Mat. Iberoam. 37}, 5 (2021), 1969--1990.

\bibitem{CeciSeis21b}
{\sc Ceci, S., and Seis, C.}
\newblock On the dynamics of point vortices for the 2{D} {E}uler equation with
  ${L}^p$ vorticity.
\newblock {\em Philos. Trans. Roy. Soc. A 380}, 2226 (2022).

\bibitem{CetroneSerafini18}
{\sc Cetrone, D., and Serafini, G.}
\newblock Long time evolution of fluids with concentrated vorticity and
  convergence to the point-vortex model.
\newblock {\em Rendiconti di Matematica e delle sue applicazioni 39\/} (2018),
  29--78.

\bibitem{Chemin96}
{\sc Chemin, J.-Y.}
\newblock A remark on the inviscid limit for two-dimensional incompressible
  fluids.
\newblock {\em Comm. Partial Differential Equations 21}, 11-12 (1996),
  1771--1779.

\bibitem{ConstantinWu95}
{\sc Constantin, P., and Wu, J.}
\newblock Inviscid limit for vortex patches.
\newblock {\em Nonlinearity 8}, 5 (1995), 735--742.

\bibitem{ConstantinWu96}
{\sc Constantin, P., and Wu, J.}
\newblock The inviscid limit for non-smooth vorticity.
\newblock {\em Indiana Univ. Math. J. 45}, 1 (1996), 67--81.

\bibitem{DavilaDelPinoMussoWei18}
{\sc Davila, J., del Pino, M., Musso, M., and Wei, J.}
\newblock Gluing methods for vortex dynamics in {E}uler flows.
\newblock {\em Arch. Ration. Mech. Anal. 235\/} (2020), 1467--1530.

\bibitem{DavilaDelPinoMussoWei20}
{\sc D\'avila, J., del Pino, M., Musso, M., and Wei, J.}
\newblock Travelling helices and the vortex filament conjecture in the
  incompressible euler equations.
\newblock {\em Preprint arXiv:2007.00606\/} (2020).

\bibitem{FengSverak15}
{\sc Feng, H., and \v{S}ver\'ak, V.~r.}
\newblock On the {C}auchy problem for axi-symmetric vortex rings.
\newblock {\em Arch. Ration. Mech. Anal. 215}, 1 (2015), 89--123.

\bibitem{Gallay11}
{\sc Gallay, T.}
\newblock Interaction of vortices in weakly viscous planar flows.
\newblock {\em Arch. Rational Mech. Anal. 200\/} (2011), 445--490.

\bibitem{Gallay12}
{\sc Gallay, T.}
\newblock Stability and interaction of vortices in two-dimensional viscous
  flows.
\newblock {\em Discr. Cont. Dyn. Systems Ser. S 5\/} (2012), 1091--1131.

\bibitem{GigaMiyakawaOsada88}
{\sc Giga, Y., Miyakawa, T., and Osada, H.}
\newblock Two-dimensional {N}avier-{S}tokes flow with measures as initial
  vorticity.
\newblock {\em Arch. Rational. Mech. Anal. 104\/} (1988), 223--250.

\bibitem{Helmholtz1858}
{\sc Helmholtz, H.}
\newblock {{\"U}ber {I}ntegrale der hydrodynamischen {G}leichungen, welche den
  {W}irbelbewegungen entsprechen}.
\newblock {\em J. Mathematik 55\/} (1858), 25--55.

\bibitem{Iftimie99}
{\sc Iftimie, D.}
\newblock \'{E}volution de tourbillon \`a support compact.
\newblock In {\em Journ\'{e}es ``\'{E}quations aux {D}\'{e}riv\'{e}es
  {P}artielles'' ({S}aint-{J}ean-de-{M}onts, 1999)}. Univ. Nantes, Nantes,
  1999, pp.~Exp. No. IV, 8.

\bibitem{JerrardSeis17}
{\sc Jerrard, R.~L., and Seis, C.}
\newblock On the vortex filament conjecture for {E}uler flows.
\newblock {\em Arch. Ration. Mech. Anal. 224}, 1 (2017), 135--172.

\bibitem{Kirchhoff76}
{\sc Kirchhoff, G.~R.}
\newblock {\em Vorlesungen \"uber mathematische Physik}.
\newblock Teubner, Leipzig, 1876.

\bibitem{FilhoNussenzveig98}
{\sc Lopes~Filho, M.~C., and Nussenzveig~Lopes, H.~J.}
\newblock An extension of {M}archioro's bound on the growth of a vortex patch
  to flows with {$L^p$} vorticity.
\newblock {\em SIAM J. Math. Anal. 29}, 3 (1998), 596--599.

\bibitem{Marchioro90}
{\sc Marchioro, C.}
\newblock On the vanishing viscosity limit for two-dimensional
  {N}avier--{S}tokes equations with singular initial data.
\newblock {\em Math. Methods Appl. Sci 12\/} (1990), 463--470.

\bibitem{Marchioro98_2}
{\sc Marchioro, C.}
\newblock On the inviscid limit for a fluid with a concentrated vorticity.
\newblock {\em Comm. Math. Phys. 196\/} (1998), 53--65.

\bibitem{Marchioro98}
{\sc Marchioro, C.}
\newblock On the localization of the vortices.
\newblock {\em Boll. Unione Mat. Ital. Sez. B Artic. Ric. Mat. (8) 1}, 3
  (1998), 571--584.

\bibitem{MarchioroPulvirenti83}
{\sc Marchioro, C., and Pulvirenti, M.}
\newblock {Euler evolution for singular initial data and vortex theory}.
\newblock {\em Comm. Math. Phys. 91}, 4 (1983), 563--572.

\bibitem{MarchioroPulvirenti93}
{\sc Marchioro, C., and Pulvirenti, M.}
\newblock {Vortices and localization in {E}uler flows}.
\newblock {\em Comm. Math. Phys. 154}, 1 (1993), 49--61.

\bibitem{Schochet96}
{\sc Schochet, S.}
\newblock The point-vortex method for periodic weak solutions of the 2{D}
  {E}uler equations.
\newblock {\em Communications on pure and applied mathematics 49\/} (1996),
  911--965.

\bibitem{Seis21}
{\sc Seis, C.}
\newblock A note on the vanishing viscosity limit in the {Y}udovich class.
\newblock {\em Canad. Math. Bull. 64}, 1 (2021), 112--122.

\bibitem{Turkington87}
{\sc Turkington, B.}
\newblock {On the evolution of a concentrated vortex in an ideal fluid}.
\newblock {\em Arch. Rational Mech. Anal. 97}, 1 (1987), 75--87.

\bibitem{Villani03}
{\sc Villani, C.}
\newblock {\em Topics in optimal transportation}, vol.~58 of {\em Graduate
  Studies in Mathematics}.
\newblock American Mathematical Society, Providence, RI, 2003.

\end{thebibliography}
	\bibliographystyle{acm}

\end{document}